\documentclass[11pt]{article}

\usepackage[T1]{fontenc}
\usepackage[utf8]{inputenc}
\usepackage{lmodern}

\usepackage[letterpaper, margin=1.15in]{geometry}
\usepackage{microtype}
\usepackage{setspace}

\usepackage{amsmath, amsthm, amssymb, amsfonts}
\usepackage{mathrsfs}

\usepackage{enumitem}

\setlist[enumerate,1]{
    label=(\arabic*),
    leftmargin=2.2em,
    itemsep=0.25em,
    topsep=0.4em
}

\usepackage[colorlinks=true,
            linkcolor=blue,
            citecolor=blue,
            urlcolor=blue]{hyperref}

\title{Anosov Diffeomorphisms of Finite-Type Surfaces}

\author{
Ra\'ul Ures\\
Department of Mathematics, Southern University of Science and Technology\\
SUSTech International Center for Mathematics\\
Shenzhen, Guangdong, China
\and
Tongyao Yu\\
School of Mathematical Sciences, Hebei Normal University\\
Shijiazhuang, Hebei, China
}

\date{\today}

\theoremstyle{plain}
\newtheorem{theorem}{Theorem}[section]
\newtheorem{lemma}[theorem]{Lemma}
\newtheorem{corollary}[theorem]{Corollary}
\newtheorem{proposition}[theorem]{Proposition}

\theoremstyle{definition}

\newtheorem{problem}[theorem]{Problem}

\theoremstyle{remark}
\newtheorem{remark}[theorem]{Remark}

\numberwithin{equation}{section}

\begin{document}

\maketitle

\begin{abstract}
We prove that a complete surface of finite topological type satisfying
Condition A is homeomorphic to the two-torus. More generally, we show that
an isolated planar end of a complete open surface satisfying Condition A
cannot be periodic under the induced action on the space of ends.
Condition A includes dense periodic points, a uniformly hyperbolic splitting into one-dimensional subbundles for a complete metric, and invariant line fields that
uniquely integrate to transverse stable and unstable foliations. These results give a partial answer to the question
of whether a complete surface satisfying Condition A must be compact.
\end{abstract}

\section{Introduction}

Throughout, by a surface we mean a second-countable surface without boundary.
The classification of surfaces supporting uniformly hyperbolic dynamics is a classical theme in differentiable dynamics.  On compact surfaces, the situation is rigid.  Let $M$ be a closed surface and let $f:M\to M$ be a $C^1$ Anosov diffeomorphism, that is, the tangent bundle admits a continuous $Df$-invariant splitting
$$
TM=E^s\oplus E^u
$$
into one-dimensional subbundles, where $E^s$ is uniformly contracted by $Df$ and $E^u$ is uniformly contracted by $Df^{-1}$.  In particular, $E^s$ and $E^u$ integrate to nonsingular stable and unstable foliations.  This already imposes a strong topological obstruction.  After passing, if necessary, to a two-fold cover on which one of the invariant line fields is orientable, one obtains a nowhere-vanishing vector field.  By the Poincaré--Hopf theorem, the Euler characteristic of this cover is zero.
Euler characteristic is multiplicative under finite covers, so $\chi(M)=0$
\cite{Milnor1965}. Hence $M$ is either the two-torus or the Klein bottle.
The latter is excluded by the classical fact that every nonsingular
one-dimensional foliation of the Klein bottle has a compact leaf, which is
incompatible with the uniform contraction or expansion of the invariant
foliations of an Anosov diffeomorphism. Thus, among closed surfaces, the only
possible surface is the two-torus.

This elementary obstruction is part of a broader rigidity picture. The classical results of Franks and Manning show that every Anosov diffeomorphism of $\mathbb T^n$ is topologically conjugate to a hyperbolic toral automorphism \cite{Franks1969,Manning1974}. In a complementary direction, Franks and Newhouse proved that every codimension-one Anosov diffeomorphism of a closed manifold is topologically conjugate to a hyperbolic toral automorphism \cite{Franks1970,Newhouse1970}.
Since on a surface both invariant bundles are one-dimensional, the compact surface case is completely rigid: a closed surface carrying an Anosov diffeomorphism must be $\mathbb T^2$, and the dynamics is modeled by linear hyperbolicity. 

The purpose of this paper is to examine to what extent this compact rigidity persists for complete noncompact surfaces.  The noncompact setting is substantially more flexible.  Hyperbolicity on an open manifold depends sensitively on the choice of metric, and complete open surfaces may support Anosov-type dynamics with behavior impossible in the compact case.  For instance, White constructed a complete Riemannian metric on $\mathbb R^2$ for which a translation becomes an Anosov diffeomorphism \cite{White1975}. However, the known examples remain far from the strong
recurrence assumptions considered here.

A natural near-example is obtained by deleting a fixed point from an Anosov
diffeomorphism of $\mathbb T^2$. The resulting punctured surface still has
dense periodic points, and the stable and unstable foliations retain much of
the qualitative structure of the original toral model away from the puncture.
However, the restricted toral metric is incomplete. Thus dense recurrence can
survive the puncturing operation, but completeness is lost exactly at the
deleted orbit. This example suggests that completeness may be the decisive
obstruction to producing complete open examples with dense recurrent
hyperbolic behavior.

A further motivation arises from the study of accessibility in three-dimensional partially hyperbolic dynamics. For a non-accessible partially hyperbolic diffeomorphism, the non-open accessibility classes may contain an invariant $su$-lamination tangent to $E^s\oplus E^u$. In the conservative setting, the boundary leaves of a proper invariant sublamination are periodic. An iterate preserving such a leaf restricts to an Anosov diffeomorphism on the leaf, and the periodic points are dense in its intrinsic topology \cite{RHRHU2008,CRHRHU2018}. Consequently, a noncompact boundary leaf gives rise naturally to an open surface carrying recurrent uniformly hyperbolic dynamics. The question of whether such leaves can occur is therefore closely connected with compactness problems for Anosov dynamics on complete surfaces.

Motivated by this phenomenon, Rodr\'iguez Hertz, Rodr\'iguez Hertz, and Ures
asked whether a complete immersed surface $L$ in a three-manifold, endowed
with Anosov dynamics for which the stable and unstable manifolds are
complete, the angle between them is bounded away from zero, the periodic
points are dense, and the stable and unstable manifolds of every periodic
point are dense, must be the two-torus \cite{CRHRHU2018}. An affirmative
answer would rule out noncompact boundary leaves in the corresponding
$su$-lamination alternative. The present paper considers an intrinsic
surface formulation of this question.

A closely related recent result is due to Ben Ovadia and DeWitt
\cite{OvadiaDeWitt2025}, who study Anosov diffeomorphisms of complete open
surfaces under substantially stronger uniformity assumptions than those used
here. Their notion of a uniform Anosov diffeomorphism imposes uniform control
on the ambient geometry, the invariant splitting, and the derivative, together
with a uniform lower bound on the angle between the invariant bundles and
leafwise completeness. Under density of periodic points they construct
Margulis measures, and show that if these measures assign infinite mass to
every unstable leaf, then the surface is closed. A topological ingredient in
their argument is that an $\mathbb R$-covered foliation without closed leaves
forces the fundamental group to be abelian
\cite[Prop.~4.3]{OvadiaDeWitt2025}. The subsequent classification argument
uses \cite[Prop.~1.2]{Mendes1977} to exclude the plane and
\cite[Lem.~4.3]{HammerlindlHRU2020} to exclude the cylinder.
On the other hand, Haefliger and Reeb showed that the leaf space of a
nonsingular one-dimensional foliation of $\mathbb R^2$ is a second-countable,
simply connected one-dimensional manifold, possibly non-Hausdorff
\cite{HaefligerReeb1957,HaefligerReebTranslation}. Consequently, if the
lifted leaf space is Hausdorff, it is homeomorphic to $\mathbb R$. Combined
with the preceding argument, this shows that in this setting Hausdorffness of
the lifted stable or unstable leaf space already forces the surface to be
$\mathbb T^2$.

We now formulate the standing hypotheses.  Let $S$ be a connected complete surface and let $f:S\to S$ be a diffeomorphism.  We say that $f$ satisfies \emph{Condition A} on $S$ if the following properties hold.

\begin{enumerate}
\item The periodic points of $f$ are dense in $S$.
\item There is a smooth complete Riemannian metric on $S$ and a continuous $Df$-invariant splitting
$$
TS=E^s\oplus E^u
$$
into one-dimensional subbundles.
\item There exist constants $C>0$ and $\lambda>1$ such that, for every $x\in S$,
every $n\ge 1$, every $v\in E^u_x\setminus\{0\}$, and every
$w\in E^s_x\setminus\{0\}$,
$$
|D_xf^n(v)|\ge C\lambda^n|v|,
\qquad
|D_xf^n(w)|\le C\lambda^{-n}|w|.
$$
\item The invariant distributions $E^s$ and $E^u$ are uniquely integrable
to invariant one-dimensional foliations $W^s$ and $W^u$.

\end{enumerate}
In the classical compact setting, the unique integrability of the invariant
distributions into stable and unstable foliations follows from the
Hadamard--Perron theory \cite{Hadamard1901,Perron1929}. Since the surface
considered here need not be compact, we include unique integrability explicitly
among the standing hypotheses.
Since $W^s$ and $W^u$ are transverse one-dimensional foliations,
they admit the usual local product charts, which we use throughout.

 The central question is the following closedness problem.

\begin{problem}
\label{prob:closedness}
Let $f:S\to S$ satisfy Condition A on a complete surface $S$.  Must $S$ be closed?
\end{problem}

If the answer is positive, then the compact classification immediately implies that $S$ is the two-torus.  Thus the essential difficulty is not the classification of compact examples, but rather the exclusion of complete open surfaces satisfying Condition A.

Our main result excludes isolated planar periodic ends. Recall that an end of a surface is planar if it has a neighborhood of genus
zero, and that an end is isolated if it is an isolated point of the end space.
In particular, an isolated planar end admits a closed neighborhood
homeomorphic to
$$
S^1\times[0,\infty).
$$
The dynamics of $f$ induces an action on the space of ends of $S$. We prove
that an isolated planar end cannot be periodic under this action.

\begin{theorem}
\label{thm:no-planar-periodic-end}
Let $f:S\to S$ satisfy Condition A on a complete open surface $S$. Then no
isolated planar end of $S$ is periodic under the action induced by $f$ on the
space of ends.
\end{theorem}

As a consequence, complete open examples satisfying Condition A cannot have finite topological type.

\begin{corollary}
\label{cor:finite-type-torus}
Let $f:S\to S$ satisfy Condition A on a complete surface $S$.  If $S$ has finite genus and finitely many ends, then $S$ is homeomorphic to $\mathbb T^2$.
\end{corollary}

Indeed, if $S$ were open with finite genus and finitely many ends, then its
genus would be contained in a compact subsurface. Hence every end of $S$
would admit a genus-zero neighborhood, so every end would be planar. Since
the end space is finite, every end is isolated. Moreover, the action of $f$
on the finite end space makes every end periodic. This contradicts
Theorem~\ref{thm:no-planar-periodic-end}. Therefore $S$ must be closed, and
the compact surface rigidity theory implies that
$$
S\cong\mathbb T^2.
$$

The proof of Theorem~\ref{thm:no-planar-periodic-end} is given in
Section~\ref{sec:planar-ends}. The argument combines the density of stable
and unstable half-leaves through periodic points with a trapping construction
near an isolated planar end. The finite-type corollary then follows
immediately.

\section{Isolated planar periodic ends}
\label{sec:planar-ends}

In this section we prove Theorem~\ref{thm:no-planar-periodic-end}. After passing, if necessary, to a finite cover on which both invariant line fields are orientable, and replacing the lifted map by a positive iterate, we assume throughout this section that $W^s$ and $W^u$ are oriented and that $f$ preserves these orientations. This causes no loss of generality. Condition A is preserved under finite covers and positive iterates, and an isolated planar periodic end lifts to finitely many isolated planar
ends, one of which is periodic for a positive iterate of the lifted map. We retain the notation $S$, $f$, $W^s$, and $W^u$ for the resulting system. 

We first record an elementary consequence of Condition A that will be used
repeatedly.

Note that the stable and unstable leaves through a periodic point are
noncompact. Indeed, suppose that a periodic point $p$ has period $m$ and
that $W^s(p)$ is compact. Then $W^s(p)$ is a circle and
$$
f^m(W^s(p))=W^s(p).
$$
On the other hand, for every sufficiently large $n$, the stable estimate
gives
$$
\operatorname{length}\bigl(f^{mn}(W^s(p))\bigr)
\le
C\lambda^{-mn}\operatorname{length}\bigl(W^s(p)\bigr)
<
\operatorname{length}\bigl(W^s(p)\bigr),
$$
a contradiction. The unstable case follows in the same way by applying
the corresponding estimate to $f^{-m}$. Since a connected one-dimensional
manifold without boundary is homeomorphic either to $S^1$ or to $\mathbb R$,
both $W^s(p)$ and $W^u(p)$ are homeomorphic to $\mathbb R$. Since the
foliations are oriented, the positive and negative half-leaves through $p$
are therefore well defined. For $\sigma\in\{s,u\}$, we denote them by
$W^\sigma_\pm(p)$.

\begin{lemma} 
\label{lem:periodic-half-leaves-dense} If $p$ is a periodic point, then each of the half-leaves $$ W^s_+(p),\qquad W^s_-(p),\qquad W^u_+(p),\qquad W^u_-(p) $$ is dense in $S$. \end{lemma}

\begin{proof}

It is enough to prove the statement for $W^u_+(p)$. The other cases follow by
applying the same argument to $f^{-1}$, and to the opposite
orientations. Replacing $f$ by a further positive iterate, we may assume that
$p$ is fixed and that $f(W^u_+(p))=W^u_+(p)$.

Let
\[
H=\overline{W^u_+(p)}.
\]
We claim that $H\setminus\{p\}$ is open in $S\setminus\{p\}$.
Let $x\in H\setminus\{p\}$, and choose a small local product
neighborhood $U$ of $x$ with $p\notin U$. Since periodic points are dense,
it is enough to show that every periodic point $q\in U$ belongs to $H$.

Choose points
\[
x_j\in W^u_+(p)\cap U,
\qquad
x_j\to x.
\]

For each $j$, the local unstable plaque through $x_j$ in $U$ is contained
in $W^u_+(p)$, since $p\notin U$. By local product structure, this plaque
meets the local stable plaque through $q$ at a point $z_j$. Hence
\[
z_j\in W^u_+(p)\cap W^s(q).
\]

If $m$ is the period of $q$, then
\[
f^{mn}(z_j)\to q
\]
as $n\to\infty$. Since $W^u_+(p)$ is $f$-invariant, it follows that
$q\in H$.

Thus $H\setminus\{p\}$ is open in $S\setminus\{p\}$. It is also closed
there, since $H$ is closed. Since $S\setminus\{p\}$ is connected and
$H\setminus\{p\}$ is nonempty,
\[
H\setminus\{p\}=S\setminus\{p\}.
\]
As $p\in H$, we conclude that
\[
H=S.
\]
Therefore $W^u_+(p)$ is dense in $S$.

\end{proof}
\begin{remark}
The proof is a standard local-product argument for hyperbolic surface
dynamics with dense periodic points. The same argument appears, for instance,
in \cite[Lem.~4.3]{HammerlindlHRU2020} in the cylinder case.
\end{remark}

\begin{proposition}
\label{prop:all-leaves-lines-complete}
Every stable and unstable leaf is homeomorphic to $\mathbb R$. Moreover,
each leaf is complete with respect to the leafwise metric induced by the
Riemannian metric on $S$.
\end{proposition}

\begin{proof}
We prove the statement for $W^s$. The unstable case is identical.

Suppose that $W^s$ has a compact leaf $L$. Since $W^u$ is oriented,
$W^s$ is transversely oriented. Choose a small tubular neighborhood $N$
of $L$ with compact closure, and a transverse interval through a point of
$L$, with coordinate $0$ corresponding to $L$. The holonomy around $L$
is represented by an orientation-preserving homeomorphism
$$
h:I_0\longrightarrow I_1,
\qquad h(0)=0,
$$
where $I_0$ and $I_1$ are intervals containing $0$. Choose a smaller interval $K$ whose closure is contained in $I_0\cap I_1$.
Then both $h$ and $h^{-1}$ are defined on $K$. We do not assume that
$h(K)\subset K$.

If $h$ is the identity on a one-sided neighborhood of $0$, its saturation
in $N$ is an open holonomy band, topologically an annulus, foliated by
compact stable leaves. By density
of periodic points, this band contains a periodic point, contradicting the
fact established above that the stable leaf of every periodic point is
homeomorphic to $\mathbb R$.

Otherwise choose a component
$$
J\subset K\setminus\operatorname{Fix}(h)
$$
sufficiently close to $0$. The endpoint of $J$ lying toward $0$, say $c$,
belongs to $K$ and satisfies $h(c)=c$. Since $h$ is increasing and has no
fixed point in $J$, either $h(t)>t$ for all $t\in J$ or $h(t)<t$ for all
$t\in J$. For $t\in J$ sufficiently close to $c$, monotonicity implies that, in the
appropriate forward or backward direction, each successive iterate lies
strictly between $t$ and $c$. Hence all these iterates are defined and
converge monotonically to $c$.

The saturation of $J$ is an open holonomy band. By density of periodic
points, choose a periodic point $q$ in this band sufficiently close to the
compact leaf corresponding to $c$. One of the stable half-leaves of $q$
converges to that compact leaf, and hence is not dense in $S$. This
contradicts Lemma~\ref{lem:periodic-half-leaves-dense}. Thus $W^s$ has no compact leaves.
Since every connected one-dimensional manifold without boundary is
homeomorphic either to $S^1$ or to $\mathbb R$, every stable leaf is
homeomorphic to $\mathbb R$.

It remains to prove leafwise completeness. Let $L$ be a stable leaf, let
$d_L$ be its leafwise distance, and let $(x_j)$ be a $d_L$-Cauchy sequence.
Since
$$
d(x_j,x_k)\le d_L(x_j,x_k),
$$
$(x_j)$ is Cauchy in the ambient metric. Completeness of $S$ gives
$x_j\to x$ for some $x\in S$.

Choose a foliation box $B$ containing $x$, and a smaller foliation box
$B_0$ containing $x$ whose closure is compact and contained in $B$. Set
$$
\rho=d\bigl(\overline{B_0},S\setminus B\bigr)>0.
$$
For all sufficiently large $j,k$,
$$
x_j,x_k\in B_0,
\qquad
d_L(x_j,x_k)<\rho.
$$
The unique stable arc in $L$ joining $x_j$ to $x_k$ therefore cannot leave
$B$, since otherwise its length would be at least $\rho$. Hence all
sufficiently large $x_j$ lie on the same stable plaque of $B$. Passing to
the limit in the foliation box gives $x\in L$, and the local leaf topology
is induced by the leafwise metric. Therefore
$$
d_L(x_j,x)\longrightarrow0.
$$
Thus $L$ is complete.

The same argument applies to $W^u$.
\end{proof}

By Proposition~\ref{prop:all-leaves-lines-complete}, every stable and unstable
leaf is an oriented copy of $\mathbb R$. We shall use this convention
throughout the remainder of the section. Thus each leaf carries its natural
linear order and has well-defined positive and negative ends, and
$W^\sigma_\pm(x)$ denotes the corresponding oriented half-leaves.

We shall also use the following elementary fact about one-dimensional
foliations.

\begin{lemma}
\label{lem:no-null-transverse-loop}
Let $\mathcal F$ be a nonsingular one-dimensional foliation on a surface.
There is no $C^1$ simple closed curve which is everywhere transverse to
$\mathcal F$ and bounds a topological disc.
\end{lemma}

\begin{proof}
Suppose that such a curve $\gamma$ bounds a disc $D$. Since $D$ is
contractible, the foliation $\mathcal F|_D$ is orientable. Transversality
implies that, after reversing the orientation if necessary, the foliation
points into $D$ along all of $\partial D$. Thus $D$ is a trapping disc,
which contradicts the Poincar\'e--Bendixson theorem for a nonsingular
foliation of the plane.
\end{proof}

We now prove the annular trapping statement needed for the end argument.
Let $e$ be an isolated planar end fixed by $f$. Since $e$ is isolated and
planar, it admits a closed annular neighborhood
$
U\cong S^1\times[0,\infty)
$ with $C^1$ boundary $C=C_0=\partial U$. Choose a nested exhaustion of $U$
by essential $C^1$ simple closed curves
$$
C=C_0,\ C_1,\ C_2,\ldots
$$
such that $C_n$ tends monotonically to the end $e$. We write $U_n$ for the
closed annulus bounded by $C_0$ and $C_n$.

After an arbitrarily small perturbation preserving the nested exhaustion, we
may assume that each $C_n$ is transverse to $W^u$ except at finitely many
isolated tangencies. We refer to this as general position with respect to
$W^u$.

For $n\ge1$, define
$$
A_n=
\left\{
x\in C_0:
\text{ there exists }y\in C_n\cap W^u_+(x)
\text{ such that }[x,y]_u\subset U_n
\right\},
$$
where $[x,y]_u$ denotes the unstable arc from $x$ to $y$ along
$W^u_+(x)$.

\begin{lemma}
\label{lem:no-trapped-unstable-arcs}
For every $n\ge 1$, no unstable half-leaf is contained in $U_n$. Moreover,
there exists $L_n>0$ such that every positively oriented unstable arc
contained in $U_n$ and starting on $C_0$ has leafwise length at most $L_n$.
\end{lemma}

\begin{proof}
Suppose first that an unstable half-leaf is contained in $U_n$. We treat the
positive half-leaf case, the other case being identical. Let
$W^u_+(x)\subset U_n$. Since \(U_n\) is compact and the unstable half-leaf is contained in
\(U_n\), the ray cannot escape every compact subset of \(U_n\).
Hence it has an accumulation point. Because the accumulation set of $W^u_+(x)$ is saturated by $W^u$, and
$\partial U_n$ contains no unstable arc by general position, this
accumulation set cannot be contained in $\partial U_n$. Hence
$W^u_+(x)$ has an accumulation point in the interior of $U_n$.

Choose a foliation box around this point and a short transversal $T$ that is
met at least twice by $W^u_+(x)$. Let $a,b\in T$ be two consecutive
intersections along the half-leaf. Denote by
$
\beta=[a,b]_u
$ the unstable segment joining them and by $\sigma\subset T$ the subarc joining
$a$ to $b$. Then
$
\Gamma=\beta\cup\sigma
$
is a simple closed curve contained in the annular neighborhood $U$.

The curve $\Gamma$ separates $S$. Indeed, if $\Gamma$ is null-homotopic in
$U$, it bounds a disc, while if it is essential in $U$, it separates the end
$e$ from the complement of $U$. Along the transversal $\sigma$, the unstable
direction crosses with constant sign, while no unstable leaf can cross the
unstable segment $\beta$. It follows that one component of
$S\setminus\Gamma$ is trapped in one unstable direction.

Choose a periodic point $p$ in this component. If the component is positively
trapped, then $W^u_+(p)$ remains in its closure. If it is negatively trapped,
the same holds for $W^u_-(p)$. In either case this contradicts
Lemma~\ref{lem:periodic-half-leaves-dense}, which asserts that both unstable
half-leaves of a periodic point are dense. Therefore no unstable half-leaf is
contained in $U_n$.

We now prove the uniform length bound. Suppose, to the contrary, that no such
$L_n$ exists. Then there are points $x_j\in C_0$ and positively oriented
unstable arcs
$
\alpha_j\subset U_n
$
starting at $x_j$ such that
$
\operatorname{length}_u(\alpha_j)\longrightarrow\infty.
$

After passing to a subsequence, $x_j\to x\in C_0$. Fix $R>0$. For all
sufficiently large $j$, the initial unstable segment of leafwise length $R$
starting at $x_j$ is contained in $\alpha_j$ and hence in $U_n$. By continuity
of the unstable foliation, these segments converge to the initial segment of
$W^u_+(x)$ of leafwise length $R$. Since $U_n$ is closed, this limiting
segment is contained in $U_n$.

As $R>0$ is arbitrary, we obtain
$
W^u_+(x)\subset U_n,
$
contradicting the first part of the proof. Hence there exists $L_n>0$ such
that every positively oriented unstable arc contained in $U_n$ and starting
on $C_0$ has leafwise length at most $L_n$.
\end{proof}

\begin{lemma}
\label{lem:gate-sets}
For every \(n\ge 1\), the set \(A_n\) is nonempty and compact. Moreover,
\[
A_{n+1}\subset A_n.
\]
In particular, the family \(\{A_n\}_{n\ge1}\) has the finite intersection
property. Consequently,
\[
A:=\bigcap_{n\ge 1} A_n
\]
is nonempty.
\end{lemma}

\begin{proof}
We first prove the nesting property. Let $x\in A_{n+1}$. By definition, there
exists $y\in C_{n+1}\cap W^u_+(x)$ such that
$
[x,y]_u\subset U_{n+1}.
$

Since $C_n$ separates $C_0$ from $C_{n+1}$ in $U$, the arc $[x,y]_u$ meets
$C_n$. Let $z$ be its first intersection with $C_n$. Then
$
[x,z]_u\subset U_n,
$
and hence $x\in A_n$. Therefore
$
A_{n+1}\subset A_n.
$

We next prove that $A_n$ is nonempty. Fix $n$. Since periodic points are dense,
choose a periodic point $p$ in the component of $U\setminus C_n$ containing
the end $e$. By Lemma~\ref{lem:periodic-half-leaves-dense}, the negative
unstable half-leaf $W^u_-(p)$ is dense and therefore meets $C_0$. Choose
$z\in C_0\cap W^u_-(p)$ and consider the oriented unstable segment from $z$
to $p$. Since $z$ and $p$ lie on opposite sides of $C_n$, this segment meets
$C_n$. Let $y$ be its first intersection with $C_n$.

Consider the connected component of the intersection of the segment from $z$
to $y$ with $U_n$ that contains $y$. Its other endpoint lies on $C_0$, since
$y$ is the first intersection with $C_n$. Denote this endpoint by $x$. Then
$
[x,y]_u\subset U_n,
$
so $x\in A_n$. Thus $A_n$ is nonempty.

It remains to prove compactness. Since $C_0$ is compact, it is enough to show
that $A_n$ is closed. Let
$$
x_j\in A_n,
\qquad
x_j\longrightarrow x\in C_0.
$$
For each $j$, choose $y_j\in C_n$ such that
$
\alpha_j=[x_j,y_j]_u\subset U_n.
$

By Lemma~\ref{lem:no-trapped-unstable-arcs}, there exists $L_n>0$ such that
$
\operatorname{length}_u(\alpha_j)\le L_n
$
for every $j$. After passing to a subsequence, we may also assume that
$y_j\to y$ for some $y\in C_n$.

Parametrize each $\alpha_j$ proportionally to leafwise arclength by
$$
\gamma_j:[0,1]\to U_n,
\qquad
\gamma_j(0)=x_j,\quad \gamma_j(1)=y_j.
$$
Let $d$ denote the ambient Riemannian distance and $d_u$ the leafwise
distance along unstable leaves. 

For every $s,t\in[0,1]$ we have
$$
d\bigl(\gamma_j(s),\gamma_j(t)\bigr)
\le
d_u\bigl(\gamma_j(s),\gamma_j(t)\bigr)
\le
\operatorname{length}_u(\alpha_j)|s-t|
\le
L_n|s-t|.
$$
Thus the maps $\gamma_j$ are uniformly Lipschitz, hence uniformly equicontinuous. Since $U_n$ is compact,
the Arzel\`a--Ascoli theorem gives a subsequence converging uniformly to a
continuous curve
$$
\gamma:[0,1]\to U_n.
$$
Moreover,
$$
\gamma(0)=x,
\qquad
\gamma(1)=y.
$$

Since $U_n$ is compact, it is covered by finitely many local product boxes.
Using the continuous dependence of unstable plaques in these boxes and the
uniform convergence of $\gamma_j$, the limit curve $\gamma$ is locally
contained in unstable plaques.
Since $\gamma(0)=x$ and $\gamma(1)=y$, its image is contained in the unstable
leaf through $x$ and joins $x$ to $y$. Hence
$$
\gamma([0,1])=[x,y]_u\subset U_n,
$$
so $x\in A_n$. Therefore $A_n$ is closed and, since $C_0$ is compact,
$A_n$ is compact.

The sets $A_n$ are nonempty, compact, and nested subsets of $C_0$. Therefore
$$
A=\bigcap_{n\ge1}A_n
$$
is nonempty.

\end{proof}

Actually, the points of \(A\) are precisely the gates through which positive unstable
rays can escape to the planar end.

\begin{lemma}
\label{lem:gates-give-proper-rays}
If $x\in A$, then the positive half-leaf
$W^u_+(x)\subset U$ is a properly embedded ray tending to the end $e$. In
particular, $W^u_+(x)$ does not accumulate in any compact subset of $U$.
\end{lemma}

\begin{proof}
Since $x\in A_n$ for every $n$, the positive unstable half-leaf
$W^u_+(x)$ reaches every $C_n$ along an unstable arc contained in $U_n$.
We first observe that
$$
W^u_+(x)\subset U.
$$
Indeed, if the half-leaf left $U$, the compact initial segment up to its first
exit could meet only finitely many of the curves $C_n$, since $C_n$ tends to
the end $e$. This contradicts the fact that $x\in A_n$ for every $n$.

We claim that $W^u_+(x)$ has no accumulation point in a compact subset of
$U$. Suppose otherwise, as in the proof of
Lemma~\ref{lem:no-trapped-unstable-arcs}, the accumulation set of
$W^u_+(x)$ cannot be contained in $\partial U$. Hence we may choose an
accumulation point
$y\in \operatorname{int}(U).$

Choose a small product box $B$ near $y$. Two sufficiently close returns of
$W^u_+(x)$ to $B$ can be joined by a short arc transverse to $W^s$. More
precisely, choose successive points $a,b$ on these returns and let
$\beta=[a,b]_u$ be the unstable segment between them. Joining $a$ to $b$
inside $B$ by a short arc $\tau$ transverse to $W^s$, after a small perturbation of the transverse arc \(\tau\), we obtain a
\(C^1\) simple closed curve
$$
\Gamma=\beta\cup\tau
$$
which is everywhere transverse to $W^s$.

By Lemma~\ref{lem:no-null-transverse-loop}, $\Gamma$ cannot be
null-homotopic. Since $\Gamma\subset U$ and $U$ is an annulus, $\Gamma$ is
essential and therefore separates off a subannulus $V\subset U$ containing
the end $e$.

Since $W^s$ is oriented and $\Gamma$ is a connected closed transversal,
the stable direction crosses $\Gamma$ with constant sign. Consequently,
one of the two oriented stable half-leaves of every point of $V$ cannot
leave $V$.

Choose a periodic point $p\in V$. One of the stable half-leaves of $p$ is
therefore contained in $V$ and cannot be dense in $S$. This contradicts
Lemma~\ref{lem:periodic-half-leaves-dense}. Hence $W^u_+(x)$ has no
accumulation point in any compact subset of $U$.

It follows that the arclength parametrization of $W^u_+(x)$ is proper.
Since $U$ is closed in $S$, it is proper also as a ray in $S$, and hence
$W^u_+(x)$ is properly embedded. 

Since the ray is proper in $U$, for every $n$ it eventually leaves the
compact annulus $U_n$. As the annuli $U_n$ exhaust $U$ toward $e$, it follows
that
$$
W^u_+(x)\longrightarrow e.
$$
\end{proof}
Next we show that there can be only finitely many such gates.

\begin{lemma}
\label{lem:finitely-many-gates}
The set \(A\subset C_0\) is finite.
\end{lemma}

\begin{proof}
Suppose that $A$ is infinite. We first observe that infinitely many distinct
unstable leaves meet $A$. Indeed, if $x,y\in A$ lie on the same unstable leaf
and $x$ precedes $y$ in the positive unstable order, then $y$ must be a
tangency of $C_0$ with $W^u$. Otherwise the transverse intersection at $y$
would force $W^u_+(x)$ to leave $U$, contradicting Lemma~\ref{lem:gates-give-proper-rays}.
Since $C_0$ has only finitely many tangencies with $W^u$, each unstable leaf
contains only finitely many points of $A$. Hence there are infinitely many
distinct gate leaves.

For each gate leaf $L$, choose the last point
$
x(L)\in A\cap L
$
in the positive unstable order. By Lemma~\ref{lem:gates-give-proper-rays},
the positive ray
$
R_L=W^u_+(x(L))
$
is properly embedded in $U$ and tends to $e$. Fix $N\ge1$. Since $R_L$
meets $C_N$ and is proper, its intersections with $C_N$ occur in a compact
initial portion of the ray. Let
$$
z_N(L)\in R_L\cap C_N
$$
be the last such intersection in the positive order. The tail of $R_L$
after $z_N(L)$ lies in the component of $U\setminus C_N$ containing $e$.

We therefore obtain infinitely many distinct pairs
$$
\bigl(x(L),z_N(L)\bigr)\in C_0\times C_N.
$$
Passing to a sequence of distinct gate leaves $L_j$, compactness of
$C_0\times C_N$ allows us, after taking a subsequence, to assume that
$$
\bigl(x(L_j),z_N(L_j)\bigr)\longrightarrow (x_\ast,z_\ast)
\in C_0\times C_N.
$$
For distinct sufficiently large $j$ and $k$, set
$$
L=L_j,\qquad L'=L_k.
$$
The corresponding points lie in common product boxes near $x_\ast$ and
$z_\ast$.
Choosing them sufficiently close, the local product
structure gives a short stable segment near $C_0$ joining $L$ to $L'$
and a second short stable segment $\tau$ near $C_N$ joining the two
gate rays.

Since the rays are proper and $\tau$ is compact, their intersections with
$\tau$ are nonempty and compact. Transversality makes these intersections discrete, hence
finite. Then let $a\in R_L\cap\tau$ and $b\in R_{L'}\cap\tau$ be the last intersections
of the two positive rays with $\tau$. 
Replacing
$\tau$ by the subarc between $a$ and $b$, we may assume that the positive
tails
$$
R_a=W^u_+(a)\subset R_L,
\qquad
R_b=W^u_+(b)\subset R_{L'}
$$
meet $\tau$ only at their initial points.

Consider the end compactification
$$
\widehat U=U\cup\{e\}\cong D^2.
$$
Since $R_a$ and $R_b$ are disjoint properly embedded rays tending to $e$,
the set
$$
R_a\cup\tau\cup R_b\cup\{e\}
$$
is a simple closed curve in $\widehat U$. Let $V$ be the component of its
complement which does not meet $C_0$. Thus $V$ is an open region extending
toward the end $e$, bounded in $U$ by the two unstable rays and the stable
arc $\tau$.

Choose a periodic point $p\in V$. The two unstable sides of $\partial V$
cannot be crossed by unstable leaves. Moreover, since $W^u$ is oriented and
$\tau$ is a connected stable arc, the positive unstable direction crosses
$\tau$ with constant sign. It follows that one of the two unstable half-leaves
$W^u_+(p)$ or $W^u_-(p)$ cannot leave $V$. This half-leaf is therefore not
dense in $S$, contradicting Lemma~\ref{lem:periodic-half-leaves-dense}.

Hence only finitely many unstable leaves can meet $A$. Since each such leaf
contains only finitely many points of $A$, the set $A$ is finite.
\end{proof}

We also need the following consequence of the preceding construction.

\begin{lemma}
\label{lem:gate-leaves}
Let
\[
\mathcal G_e^+
=
\left\{
L\in W^u:
\text{ the positive end of }L\text{ tends to }e
\right\}.
\]
Then
\[
\mathcal G_e^+
=
\{W^u(a):a\in A\}.
\]
In particular, $\mathcal G_e^+$ is finite and nonempty.
\end{lemma}

\begin{proof}
If $a\in A$, then by Lemma~\ref{lem:gates-give-proper-rays},
the positive half-leaf $W^u_+(a)$ tends properly to $e$. Hence
$W^u(a)\in\mathcal G_e^+$.

Conversely, let $L\in\mathcal G_e^+$. By definition, the positive end of
$L$, with respect to the fixed orientation of $W^u$, tends to $e$. 
We claim that $L$ meets $C_0$. Suppose not. Since the positive end of $L$
tends to $e$, a positive tail of $L$ is contained in $U$. As $L$ does not meet
$\partial U=C_0$, the whole leaf $L$ is contained in $U$.

The negative end of $L$ cannot accumulate in a compact subset of $U$. Indeed,
such an accumulation would give, exactly as in the proof of
Lemma~\ref{lem:gates-give-proper-rays}, two nearby points on $L$ joined
by a short arc transverse to $W^s$. Together with the unstable subarc between
them, this gives a closed curve transverse to $W^s$.

If this curve is null-homotopic, Lemma~\ref{lem:no-null-transverse-loop}
gives a contradiction. If it is essential, the essential-case argument in
Lemma~\ref{lem:gates-give-proper-rays} produces a stable half-leaf of a
periodic point trapped on one side of the curve.
This contradicts Lemma~\ref{lem:periodic-half-leaves-dense}.

Thus $L$ is a properly embedded unstable line in $U$ whose two ends both tend
to $e$. After adding the end $e$, the compactification $U\cup\{e\}$ is a
closed disc, and $L\cup\{e\}$ separates it into two components. Let $V$ be the component disjoint from $C_0$, and choose a periodic point
$p\in V$. Since
unstable leaves cannot cross $L$, at least one unstable half-leaf of this
periodic point remains trapped in that component. This contradicts
Lemma \ref{lem:periodic-half-leaves-dense}. Therefore $L\cap C_0\neq\varnothing$.

Since the positive end of $L$ tends to $e$, there is a last point
$a\in L\cap C_0$ in the positive order on $L$. For every $n$, the first
intersection $y_n$ of the positive ray from $a$ with $C_n$ satisfies
$
[a,y_n]_u\subset U_n.
$

Hence $a\in A_n$ for every $n$, and therefore $a\in A$.
Since $A$ is finite by Lemma~\ref{lem:finitely-many-gates}, the set
$\mathcal G_e^+$ is finite.
\end{proof}

\begin{proof}[Proof of Theorem~\ref{thm:no-planar-periodic-end}]
Assume, for contradiction, that $e$ is an isolated planar periodic end. Replacing $f$ by a further positive iterate, we may assume that $e$ is fixed.

By Lemma~\ref{lem:gate-leaves}, the set
$
\mathcal G_e^+=\{W^u(a):a\in A\}
$
is finite and nonempty. 

Since $f$ fixes $e$ and preserves the orientation of $W^u$,
\[
f(\mathcal G_e^+)=\mathcal G_e^+.
\]

We do not claim that $f(A)=A$, since $f(C_0)$ need not equal $C_0$.

Therefore for any leaf $L\in\mathcal G_e^+$, $L$ is periodic, so there exists $k>0$ such
that
$$
f^k(L)=L.
$$

Replacing $k$ by a multiple, we may assume that $C^{-1}\lambda^{-k}<1$, where
$C$ and $\lambda$ are the constants in the unstable expansion estimate. Then
$f^{-k}|_L$ is a strict contraction with respect to the unstable leafwise
metric on $L$. 
By Proposition~\ref{prop:all-leaves-lines-complete}, the unstable leaf $L$ is
complete in its leafwise metric. Therefore the Banach fixed point theorem
gives a point $p\in L$ such that
$$
f^{-k}(p)=p.
$$

Equivalently,
$$
f^k(p)=p.
$$
Thus $p$ is a periodic point.

But $L\in\mathcal G_e^+$, so one unstable half-leaf of $p$ tends properly to
the end $e$. This half-leaf is eventually contained in every sufficiently
small neighborhood of $e$, and therefore is not dense in $S$. This contradicts
the density of unstable half-leaves through periodic points.

Hence no isolated planar end can be periodic.
\end{proof}

\end{document}